\documentclass[letterpaper,oneside]{article}
\usepackage[english]{babel}
\usepackage[utf8]{inputenc}
\usepackage{amsmath,amssymb,amsthm}

\usepackage{pgffor}

\usepackage{bbm,bm}
\usepackage{enumitem}
\usepackage{calligra,mathrsfs} 
\usepackage{quiver}
\usepackage[normalem]{ulem}
\usepackage{stmaryrd} 
\usepackage[linesnumbered,commentsnumbered,ruled,vlined]{algorithm2e}

\allowdisplaybreaks

\newlist{steps}{enumerate}{1}
\setlist[steps]{label=\underline{Step \arabic*:},ref=\arabic*,wide,nosep}

\usepackage{latexsym,mathtools,braket}
\usepackage{tikz-cd}
\usepackage{thmtools,thm-restate}
\usepackage{cancel}

\usepackage[colorlinks,plainpages,backref]{hyperref}
\usepackage{cleveref} 

\numberwithin{equation}{section}

\declaretheorem[name=Theorem,
	refname={theorem,theorems},
	Refname={Theorem,Theorems},
	numberwithin=section
	]{theorem}
\declaretheorem[name=Question,
	refname={question,questions},
	Refname={Question,Questions},
	numberwithin=section
	]{question}

\declaretheorem[name=Proposition,
	refname={proposition, propositions},
	Refname={Proposition, Propositions},
	sibling=theorem]{proposition}
\declaretheorem[name=Lemma,
	refname={lemma,lemmas},
	Refname={Lemma,Lemmas},
	sibling=theorem]{lemma}

\declaretheorem[name=Remark,
	refname={remark,remarks},
	Refname={Remark,Remarks},
	style=remark,
	sibling=theorem]{remark}
\declaretheorem[name=Definition,
	refname={definition,definitions},
	Refname={Definition,Definitions},
	style=remark,
	sibling=theorem]{definition}
\declaretheorem[name=Example,
	refname={example,examples},
	Refname={Example,Examples},
	sibling=theorem,
	style=remark]{example}

\newcounter{claim}[theorem]

\renewcommand{\vec}[1]{\mathbf{#1}}

\newcommand{\CC}{\mathbb{C}}
\newcommand{\ZZ}{\mathbb{Z}}
\newcommand{\NN}{\mathbb{N}}
\newcommand{\RR}{\mathbb{R}}

\newcommand{\PP}{\mathbb{P}}

\DeclareMathOperator{\shRHom}{R\mathcal{H}\mathit{om}}

\DeclareMathOperator{\Ch}{Ch}

\newcommand{\an}{\mathrm{an}}

\newcommand{\gr}{\mathrm{gr}}

\DeclareMathOperator{\Irr}{\sh{I}\mathit{rr}}
\DeclareMathOperator{\init}{in}
\DeclareMathOperator{\Div}{Div}

\DeclareMathOperator{\Supp}{Supp}

\DeclareMathOperator{\rank}{rank}

\DeclareMathOperator{\im}{im}

\DeclareMathOperator{\ord}{ord}

\DeclareMathOperator{\Sing}{Sing}

\DeclareMathOperator{\irrmult}{\mu^{irr}}
\DeclareMathOperator{\irrdiv}{ID}

\newcommand{\uppercaseFactory}[2]{
    \foreach \x in {A,...,Z}{
        \expandafter\xdef\csname #1\x\endcsname{#2{\x}}}}
\newcommand{\lowercaseFactory}[2]{
    \foreach \x in {a,...,z}{
        \expandafter\xdef\csname #1\x\endcsname{#2{\x}}}}

\newcommand{\sh}{\mathcal}
\uppercaseFactory{sh}{\noexpand\sh}
\uppercaseFactory{cal}{\noexpand\mathcal}
\uppercaseFactory{scr}{\noexpand\mathscr}
\uppercaseFactory{mod}{}
\uppercaseFactory{fr}{\noexpand\mathfrak}
\lowercaseFactory{fr}{\noexpand\mathfrak}
\uppercaseFactory{bf}{\noexpand\mathbf}
\lowercaseFactory{bf}{\noexpand\mathbf}
\uppercaseFactory{rm}{\noexpand\mathrm}
\lowercaseFactory{rm}{\noexpand\mathrm}

\title{A new formulation of regular singularity}
\author{Avi Steiner}
\date{\today}

\begin{document}

\maketitle

\begin{abstract}
    We provide an alternative definition for the familiar concept of regular singularity for meromorphic connections. Our new formulation does not use derived categories, and it also avoids the necessity of finding a special good filtration as in the formulation due to Kashiwara--Kawai. Moreover, our formulation provides an explicit algorithm to decide the regular singularity of a meromorphic connection. An important intermediary result, interesting in its own right, is that taking associated graded modules with respect to (not necessarily canonical) $V$-filtrations commutes with non-characteristic restriction. This allows us to reduce the proof of the equivalence of our formulation with the classical concept to the one-dimensional case. In that situation, we extend the well-known one-dimensional Fuchs criterion for ideals in the Weyl algebra to arbitrary holonomic modules over the Weyl algebra equipped with an arbitrary $(-1,1)$-filtration.
\end{abstract}

\section{Introduction}

Fuchs famously proved in 1866 that, for a linear ordinary differential operator $P=\sum_{i=0}^r a_i(x)\partial^i$ ($a_i$ holomorphic near $x=p$ with $a_r$ not identically zero), the following conditions are equivalent (see, e.g., \cite[Th.~5.1.5]{htt} or \cite[Th.~1.4.18]{sst} for a precise statement in the context of $D$-modules, or \cite[\S15.3]{ince} for a proof):
\begin{enumerate}[label=\textnormal{(\alph*)}]
    \item $\ord_{x=p}(a_i/a_r)\geq -(r-i)$ for $i=0,\ldots,r$.
    \item Every multivalued solution near $p$ to the ordinary differential equation $Pu=0$ can be expressed as a linear combination of functions of the form
    \[x^\lambda g(x)(\log x)^k,\]
    where $\lambda\in \CC$ and $g(x)$ is holomorphic near $p$.
\end{enumerate}
A linear ordinary differential operator $P$ which satisfies these equivalent conditions is said to be \emph{regular singular at $p$}.
It was later realized by Deligne \cite{deligneRH} that this regularity condition (and its higher-dimensional generalization) has significant geometric meaning, and is exactly the condition needed to relate the so-called de Rham cohomology of a flat vector bundle $\shE$ on a complex \emph{algebraic} variety $X$ with the cohomology of $X$ with coefficients in the horizontal sections of $\shE$. See \cite{Meb04} for a thorough discussion of the history of regularity in the algebro-geometric setting. 

Various authors have generalized regularity to higher dimensions, such as Deligne \cite{deligneRH}, Mebkhout \cite{Meb89}, and Kashiwara--Kawai \cite{KKshort} (see also the presentation in \cite[\S\S5-6]{htt}). However, although these generalizations are equivalent, they are all non-algorithmic. It is therefore natural to ask for an alternative (equivalent) definition of regularity which \emph{is} algorithmic. We accomplish this goal in \Cref{th:const-implies-rh}. Our definition has the advantage of not using derived categories, and avoids the necessity of finding a special good filtration as in \cite{KKshort}.

In certain instances, we will need results from more advanced sources such as \cite{htt}, but we are for the most part able to avoid the formalism of derived categories except in very specific instances. Various key concepts and notation, such as $\Irr_Y(\shM^\an)$, $\Ch(\shM)$, $T^*_XY$, local coordinates, good filtrations,etc., are defined in \S\ref{sec:notation}.

To motivate our approach, we recall the following theorem, which is essentially \cite[Th.~2.5.1]{sst}:

\begin{theorem}[{\cite[Th.~2.5.1]{sst}}]\label{th:rh-implies-const}
    Let $X$ be a smooth variety, $\mathcal{M}$ be a regular holonomic $\shD_X$-module, $p\in X$, and let $x_1,\ldots,x_n$ be local coordinates centered at $p$. Let $I$ be any ideal of the $n$th Weyl algebra $D_n=\CC\langle x,\partial\rangle$ such that $\mathcal{M}\cong \mathcal{D}_X/\mathcal{D}_X I$ on a neighborhood of $p$.\footnote{Such an ideal exists because $\shM$ is holonomic, see \cite[Th.~2.5]{coutinho}.} Then for all weights $w\in \RR^n$,
    \begin{equation}\label{eq:grobdef-rank}
        \rank\left(D_n/D_n\cdot \init_{(-w,w)}(I)\right) = \rank(\mathcal{M}).
    \end{equation}
\end{theorem}

The converse is in general false: take any irregular holonomic $\mathcal{D}_X$-module supported on a proper subset of $X$. Then by the semicontinuity of holonomic rank (\cite[Th.~2.2.1]{sst}), both sides of \eqref{eq:grobdef-rank} vanish and hence are trivially equal for all $w\in \RR^n$. 

The next best thing to try, therefore, is when $\mathcal{M}$ is a meromorphic connection (see \S\ref{sec:mero} for the definition). The goal of this article is to prove the converse of \Cref{th:rh-implies-const} in this case, which for convenience we state as an equivalence. Below, \[V_\bullet^{\{p\}}\shD_X \coloneqq \Set{P\in \shD_X | P(\frm_p^i) \subseteq \frm_p^{i-\bullet}\text{ for all }i}\]
and $\frm_p$ is the maximal ideal at $p$.

\begin{theorem}\label{th:const-implies-rh}
    Let $X$ be a smooth \emph{complete} variety. Let $\mathcal{M}$ be a meromorphic connection on $X$. The following are equivalent:
    \begin{enumerate}[label=\textnormal{(\alph*)}]
        \item\label{th:const-implies-rh.reg} $\mathcal{M}$ is regular.
        \item\label{th:const-implies-rh.Dn} For all $p\in X$, there exist local coordinates $x_1,\ldots,x_n$ centered at $p$ and an ideal $I$ of $D_n$ with $\mathcal{M}\cong \mathcal{D}_X/\mathcal{D}_X I$ on a neighborhood of $p$ such that for all weights $w\in \RR^n$, 
        \begin{equation}\label{eq:comparison}
            \rank\left(D_n/\init_{(-w,w)}(I)\right) = \rank(\mathcal{M}).
        \end{equation}
        \item\label{th:const-implies-rh.irreds} Let $Y_1,\ldots,Y_r$ be the irreducible components of the pole divisor of $\shM$. For all $j$, for a general point $p\in Y_j$, and for any (equivalently some) $V^{\{p\}}_\bullet\shD_X$-good filtration $U_\bullet\shM$ defined near $p$,
    \[ \rank (\gr^U(\shM)) = \rank(\shM).\]
    \end{enumerate}
\end{theorem}

Actually, we are going to prove the following stronger statement, of which \Cref{th:const-implies-rh} will be a corollary. See \eqref{eq:M(*)} for the definition of $\shM(*Y)$.

\begin{theorem}\label{th:Irr-is-0}
    Let $\shM$ be a holonomic $\shD_X$-module, $Y\subseteq X$ a hypersurface containing $\Sing(\shM)$. Let $\{S_j\}$ be a stratification of $Y$ such that $\Ch(\shM)\cup \Ch(\shM(*Y))$ is contained in $T^*_XX\cup \bigcup_j \overline{T^*_{S_j}X}$. Then $\Irr_Y(\shM)$ has no cohomology if and only if for each $S_j$ with $\dim S_j=\dim Y$, for any (equivalently some) $p\in S_j$, and for any (equivalently some) $V^{\{p\}}_\bullet\shD_X$-good filtration $U_\bullet\shM$ defined near $p$,
    \[ \rank (\gr^U(\shM)) = \rank(\shM).\]
\end{theorem}

We now use \Cref{th:Irr-is-0} to prove \Cref{th:const-implies-rh}.
\begin{proof}[Proof of \Cref{th:const-implies-rh}]
    (\ref{th:const-implies-rh.reg}$\iff$\ref{th:const-implies-rh.irreds}) Since $\shM$ is meromorphic with pole divisor $Y$, it is regular if and only if $\Irr_Y(\shM)$ has no cohomology (\cite[Cor.~4.3-14]{Meb04}). Now apply \Cref{th:Irr-is-0}.

    \smallskip
    (\ref{th:const-implies-rh.reg}$\implies$\ref{th:const-implies-rh.Dn}) This is \cite[Th.~2.5.1]{sst} as presented in \Cref{th:rh-implies-const}.

    \smallskip
    (\ref{th:const-implies-rh.Dn}$\implies$\ref{th:const-implies-rh.irreds}) Let $p$ be a general point of $Y_j$. Find local coordinates centered at $p$ and an ideal $I$ of $D_n$ as in the statement of \ref{th:const-implies-rh.Dn}. Equip $D_n/I$ with the $V_\bullet D_n$-filtration with respect to the weight $(-w,w)=(-1,\ldots,-1,1,\ldots,1)$. The $V_\bullet D_n$ filtration with respect to this weight is exactly $V^{\{p\}}_\bullet\shD_X \cap D_n$. So, $\shM\cong\shD_X/\shD_X I$ has the filtration induced by $V^{\{p\}}_\bullet\shD_X$, and 
    \[ \gr^U(\shM) \cong \shD_X\otimes_{D_n}\gr^U(D_n/I) \cong \shD_X\otimes_{D_n}(D_n/\init_{(-w,w)}(I)).\]
    Now use that tensoring with $\shD_X$ doesn't change the rank, and apply \ref{th:const-implies-rh.Dn}.
\end{proof}

\begin{remark}
    At first glance, \Cref{th:Irr-is-0} seems to contradict \cite[Cor.~4.9 and Th.~6.4]{nilssonAhyp}, which imply that, for an irregular $A$-hypergeometric system $\shM_A(\beta)$ equipped with the induced $V^{\{0\}}_\bullet\shD_X$-filtration, $\gr^U(\shM_A(\beta))$ and $\shM_A(\beta)$ have the same rank. The subtlety is that the origin never satisfies the genericity condition in the statement of \Cref{th:Irr-is-0}; in particular, $0$ is contained in every irreducible component of the singular locus of $\shM_A(\beta)$, which follows from \cite[Prop.~3.8 and Lem.~3.14]{SW-irreg-gkz}.
\end{remark}



\subsection{Outline}
In \S\ref{sec:notation}, we make explicit some notation and conventions that will be used throughout the article.

In \S\ref{sec:examples}, we collect various examples exhibiting the main results.

In \S\ref{sec:Fuchs}, we prove a slight generalization of Fuchs' Theorem (\cite[Th.~1.4.18]{sst}) which applies to arbitrary (holonomic) $D_1$-modules with arbitrary good $V_\bullet D_1$-filtrations rather than just cyclic ones.

In \S\ref{sec: restr-and-gr}, we prove the compatibility of non-characteristic restriction with taking associated graded modules. This will allow us, in \S\ref{sec:proof}, to reduce to the one-dimensional case, where we can use the result of \S\ref{sec:Fuchs}.

In \S\ref{sec:Nilsson}, we prove a higher-dimensional version of Fuchs' criterion, namely that a meromorphic connection on a complete variety has rank many Nilsson solutions in every direction at every point (these notions are made precise in this section).

In \S\ref{sec:algorithm}, we adapt \Cref{th:Irr-is-0} to provide an algorithm to compute the support of $\Irr_Y(\shM)$, and then an algorithm to decide whether a given meromorphic connection on a complete variety is irregular.

Finally, in \S\ref{sec:irrdiv}, we collect the main result into a statement about divisors and divisor classes.

\subsection{Acknowledgements}
We would like to thank Christine Berkesch, Mohamed Barakat, Maria-Cruz Fern\'andez-Fern\'andez, Francisco Castro-Jimenez, Anna-Laura Sattelberger, Christian Sevenheck, and Uli Walther for useful discussions; Claude Sabbah for his assistance with understanding the irregularity complex; and Saiei-Jaeyong Matsubara-Heo and Henry Dakin for their help with \S\ref{sec:Fuchs} and \S\ref{sec: restr-and-gr}, respectively.

\section{Notation and conventions}\label{sec:notation}
\subsection{Varieties}
All varieties are complex and considered with the Zariski topology. Sometimes, however, we will need to consider their analytifications, in which case we will denote the analytification functor by $(-)^\an$.

The conormal bundle of a smooth variety X along a smooth subvariety $Y$ is denoted by $T^*_YX$.

\subsubsection{Local coordinates}
When we say that $x_1,\ldots,x_n$ are local coordinates centered at a point $p$ of a smooth variety $X$, we mean the following: $x_1,\ldots,x_n$ are regular functions on an affine open neighborhood $U$ of $p$ which generate the maximal ideal $\frm_p$ at $p$, there are vector fields $\partial_1,\ldots,\partial_n$ on $U$, and these satisfy
\begin{equation}
    \begin{cases}
        [\partial_i,\partial_j]=0 \text{ and } \partial_i(x_j) = \delta_{ij} \quad (i,j=1,\ldots,n),\\
        \Theta_U = \bigoplus_{i=1}^n \shO_V\partial_i.
    \end{cases}
\end{equation}
Here, $\Theta_U$ is the sheaf of vector fields on $U$. Although $\partial_1,\ldots,\partial_n$ are technically part of the data of choosing local coordinates, we will always leave them implicit. See \cite[Ch.~A.5]{htt} for more details.

\subsection{\texorpdfstring{$D$}{D}-module concepts}\label{sec:dmod-concepts}
The sheaf of linear partial differential operators on a smooth variety $X$ is denoted $\shD_X$. The $n$th Weyl algebra is denoted $D_n$. The characteristic variety of a coherent $\shD_X$-module $\shM$ is denoted $\Ch(\shM)$.

If $f\colon X\to Y$ is a morphism of smooth varieties and $\shM$ is a $\shD_Y$-module, we denote by $f^*\shM$ the $\shO$-module pullback of $\shM$ equipped with its canonical $\shD_X$-module structure. See \cite[\S1.3]{htt} for details. When $i\colon \CC^r\hookrightarrow\CC^n$ is the inclusion of a coordinate subspace and $M$ is a $D_n$-module, we abuse notation and write $i^*M$ instead of sheafifying, applying $i^*$, then taking global sections. 

If $Z\subseteq X$ is a hypersurface, denote by $\shO_X(*Z)$ the sheaf of rational functions on $X$ with poles only along $Z$, i.e.\ $\shO_X(*Z) = \bigcup_{k\in \NN} \shO_X(-kZ)$. If $\shM$ is a $\shD_X$-module, set
\begin{equation}\label{eq:M(*)}
    \shM(*Z) \coloneqq \shM\otimes_{\shO_X} \shO_X(*Z).
\end{equation}
This is canonically also a $\shD_X$-module. Since we are working in the algebraic category, $\shM(*Z)\cong j_*j^*\shM$, where $j\colon X\setminus Z\hookrightarrow X$ is inclusion. Note that this is not true in the analytic category, since $j_*j^*\shM^\an$ is not even a coherent $\shD_{X^\an}$-module.

The \emph{rank} (or \emph{holonomic rank}) of a coherent $\shD_X$-module $\shM$ is
\begin{equation}
    \rank(\shM) \coloneqq \dim_{\CC(X)} (\CC(X)\otimes_{\shO_X}\shM),
\end{equation}
where $\CC(X)$ denotes the field of rational functions of $X$. If $p$ is a general point of $X$ and $\shM\cong \shD_X/\shI$ near $p$ with $\shI$ a $\shD_X$-ideal, then 
\begin{equation}
    \rank(\shM) = \dim_\CC\Set{f\in \shO_{X^\an,p} | Pf = 0 \text{ for all }P\in \shI}.
\end{equation}
See the discussion of rank in \cite[\S1.4]{sst}, in particular \cite[Th.~1.4.19]{sst}.

The \emph{singular locus} of a coherent $\shD_X$-module $\shM$ is
\begin{equation}
    \Sing(\shM) \coloneqq \pi(\Ch(\shM)\setminus T^*_XX),
\end{equation}
where $T^*_XX$ is the zero section of the cotangent bundle of $X$, and $\pi\colon T^*X\to X$ is the bundle projection.

\subsubsection{Meromorphic connections}\label{sec:mero}
Let $X$ be a smooth variety. A \emph{meromorphic connection} $\shM$ with poles along a hypersurface $Y\subseteq X$ is a holonomic $\shD_X$-module such that, if $f$ is any local defining equation for $Y$, multiplication by $f$ acts invertibly on $\shM$, i.e.\ such that $\shM\cong \shM(*Y)$. The hypersurface $Y$ is called the \emph{pole divisor} of $\shM$. 

More concretely, a meromorphic connection is a $\shD_X$-module (\'etale-)locally given by a system of PDEs
\[\left\{
\begin{aligned}
    \frac{\partial \vec{u}}{\partial x_1} &= A_1(x)\vec{u}\\
    &\ldots\\
    \frac{\partial \vec{u}}{\partial x_n} &= A_n(x)\vec{u},\\
\end{aligned}\right.\]
where $\vec{u}$ is a vector of dependent variables, and $A_i(x)$ is a matrix of \emph{rational} functions. The pole divisor is then the union of the poles of the $A_i(x)$'s. 

Finally, one more equivalent definition of a meromorphic connection with pole divisor $Y$ is a $\shD_X$-module which is also a locally free $\shO_X(*Y)$-module of finite rank.

When $\shM$ is a meromorphic connection with pole divisor $Y$, then $\Sing(\shM)=Y$.

\subsubsection{Regularity}
There are many equivalent definitions of regularity. The definition that has worked best for us in this paper is due to Mebkhout \cite[D\'ef.~4.2-3]{Meb04}: 

Let $Y\subseteq X$ be a hypersurface, $i\colon Y\hookrightarrow X$ inclusion. The \emph{irregularity complex with respecto to $Y$} of a holonomic $\shD_{X^\an}$-module $\shM$ is\footnote{Mebkhout denotes this $\operatorname{Irr}_Y^*(\shM)$ in \cite{Meb04}.}
\begin{equation}\label{eq:alt-Irr}
    \Irr_Y(\shM) \coloneqq i^{-1}\shRHom_{\shD_{X^\an}}(\shM(*Y), \shO_{X^\an}).
\end{equation}
This is a perverse sheaf supported on $Y$ (\cite[Rem.~3.5-15]{Meb04}), however we won't need this fact. In fact, we will need the explicit definition of $\Irr_Y(\shM)$ itself only once, in the proof of \Cref{Irr-vs-restr}\ref{Irr-vs-restr.irr}.


Note that by \cite[Cor.~3.4-4]{Meb04} an equivalent definition of $\Irr_Y(\shM)$ is 
\begin{equation}
    \Irr_Y(\shM) \cong i^{-1}\shRHom_{\shD_{X^\an}}(\shM, \shO_{\widehat{X|Y}}/\shO_{X^\an})[-1],
\end{equation}
where $\shO_{\widehat{X|Y}}$ is the formalization of $\shO_{X^\an}$ along $Y$.

\begin{definition}
    \begin{enumerate}[label=\textnormal{(\alph*)}]
        \item A holonomic $\shD_{X^\an}$-module $\shM$ is \emph{regular} if $\Irr_Y(\shM)$ has no cohomology for every hypersurface $Y$ of $X$.\footnote{Mebkhout's original definition is for every \emph{subvariety} $Y$ (using a slightly modified definition of $\Irr_Y(\shM)$), but he then remarks that it is equivalent to consider only hypersurfaces.}
        \item Let $\overline{X}$ be any smooth completion of $X$, $j\colon X\hookrightarrow\overline{X}$ inclusion. A holonomic $\shD_X$-module $\shM$ is \emph{regular} if $(j_*\shM)^\an$ is regular. Note that this is independent of the choice of smooth completion.
    \end{enumerate}
\end{definition}
\begin{remark}
    One reason that the definition of regularity in the algebraic category is different from in the analytic category is as follows: Let $\shE\neq0$ be an (algebraic) integrable connection on $X$, i.e.\ $\shE$ is locally free as an $\shO_X$-module. Although $j_*\shE^\an$ is a $\shD_{X^\an}$-module, it is never a coherent $\shD_{X^\an}$-module. By Deligne's Riemman--Hilbert Correspondence (\cite[Th.~5.2.20]{htt}), it turns out that there is \emph{always} a unique \emph{regular} analytic meromorphic connection $\shM\subseteq j_*\shE^\an$ such that $j^{-1}\shM = \shE^\an$. On the other hand, although $(j_*\shE)^\an$ is a meromorphic connection contained in $j_*\shE^\an$, it is equal to $\shM$ if and only if $\shE$ is regular.
\end{remark}
\medskip

This definition simplifies significantly when $\shM$ is a meromorphic connection on a complete variety.
\begin{theorem}[{\cite[Cor.~4.3-14]{Meb04}}]
    A meromorphic connection $\shM$ on a smooth \emph{complete} variety $X$  with pole divisor $Y$ is regular if and only if the complex $\Irr_Y(\shM^\an)$ has no cohomology.
\end{theorem}

For the convenience of the reader, we also include the following useful fact about the category of regular holonomic $\shD_X$-modules.

\begin{theorem}[{\cite[Th.~4.2--4]{Meb04}}]
    The category of regular holonomic $\shD_X$-modules is a full abelian subcategory of the category of holonomic $\shD_X$-modules which is closed under extensions.
\end{theorem}

\subsubsection{Non-characteristic restrictions}
\begin{definition}
    Let $i\colon Z\hookrightarrow X$ be the inclusion of a smooth subvariety into the smooth variety $X$, and let $\shM$ be a coherent $\shD_X$-module. We say that $i$ is \emph{non-characteristic for $\shM$} if 
    $T^*_Z X \cap \Ch(\shM) \subseteq T^*_XX$.
\end{definition}

This is a special case of a more general definition of non-characteristic which applies to arbitrary morphisms $f\colon Z\to X$ between smooth varieties (see \cite[Def.~2.4.2]{htt}); however, we won't need the general version in this article.

The main reason for us to consider non-characteristic restrictions is that, as a special case of the Cauchy--Kowalevski--Kashiwara theorem (\cite[Th.~4.3.2]{htt}), $\rank(\shM)= \rank(i^*\shM)$ when $i$ is non-characteristic for $\shM$. 

\subsection{Filtrations and weights}
\subsubsection{Filtrations in general}
If $(R, F_\bullet R)$ is a filtered ring, we denote by $\gr^F R$ the associated graded ring with respect to this filtration. Similarly, if $(M, F_\bullet M)$ is a filtered module over the filtered ring $(R, F_\bullet R)$, we denote by $\gr^F M$ the associated graded $\gr^F R$-module.

\begin{definition}
    If $(M, F_\bullet M)$ is a filtered module over a filtered ring $(R, F_\bullet R)$, we say that the filtration on $M$ is a \emph{good filtration} if there exist $m_1,\ldots,m_s\in M$ and $k_1,\ldots,k_s\in \ZZ$ such that for all $p\in \ZZ$, 
    \[F_pM = F_{p-k_1}R\cdot m_1 + \cdots F_{p-k_s}R\cdot m_s.\]
\end{definition}

\subsubsection{\texorpdfstring{$V$}{V}-filtrations and weights}
Given a point $p$ in a smooth variety $X$, define
\begin{equation}
    V_\bullet^{\{p\}}\shD_X \coloneqq \Set{P\in \shD_X | P(\frm_p^i) \subseteq \frm_p^{i-\bullet}\text{ for all }i},
\end{equation}
where $\frm_p$ is the maximal ideal at $p$. 

Given a $w\in \RR^n$, the \emph{$(-w,w)$-order} of $0\neq P=\sum_{\alpha,\beta} c_{\alpha\beta}x^\alpha\partial^\beta\in D_n$ is 
\begin{equation}
    \ord_{(-w,w)}(P) \coloneqq \max\Set{w\cdot(\beta-\alpha) | c_{\alpha\beta}\neq 0},
\end{equation}
the \emph{$(-w,w)$-initial form} of $0\neq P=\sum_{\alpha,\beta} c_{\alpha\beta}x^\alpha\partial^\beta\in D_n$ is 
\begin{equation}
    \init_{(-w,w)}(P) \coloneqq \sum_{w\cdot(\beta-\alpha) = \ord_{(-w,w)}(P)} c_{\alpha\beta}x^\alpha\partial^\beta,
\end{equation}
and the filtration \emph{$V_\bullet D_n$ with respect to $w$} is 
\begin{equation}
    V_\bullet D_n \coloneqq \Set{P \in D_n | \ord_{(-w,w)}(P)\leq \bullet}.
\end{equation}

We remark that there are isomorphisms $\gr^{V^{\{p\}}}\shD_X\cong \shD_X$ (near $p$) and $\gr^VD_n\cong D_n$. 

If $\shM$ (resp.\ $M$) is a $\shD_X$-module (resp.\ a $D_n$-module), we will always denote by $U_\bullet \shM$ (resp.\ $U_\bullet M$) a filtration with respect to $V_\bullet^{\{p\}}\shD_X$ (resp.\ $V_\bullet D_n$). The particular point $p$ or weight $w$ will be clear from the context.

\section{Examples}\label{sec:examples}
\begin{example}\label{ex:elemirr}
    Let $X$ be a smooth variety of dimension $n$, $f$ any rational function on $X$. Let $\shE^f$ be the rank one meromorphic connection given by $d+df\wedge$. One calls $\shE^f$ an \emph{elementary irregular meromorphic} connection, and its pole divisor is exactly the set $Y$ of poles of $f$. It is well-known that $\Supp(\Irr_Y(\shE^f))=Y$. We give an alternative proof of this fact using \Cref{th:Irr-is-0}. 

    For simplicity, let us assume that $Y$ is irreducible. Let $p$ be a general point of $Y$. Choose coordinates $x_1,\ldots,x_n$ centered at $p$, and write $f=a/b$, where $a,b$ are regular near $p$. Then, near $p$,
    \[\shE^f = \frac{\shD_X}{\sum_{i=1}^n\shD_X \left(b^2\partial_i - b^2\frac{\partial f}{\partial x_i}\right)}=\frac{\shD_X}{\sum_{i=1}^n \shD_X \left(b^2\partial_i - b\frac{\partial a}{\partial x_i} + a\frac{\partial b}{\partial x_i}\right)}.\]
    Equip $\shE^f$ with the $V^{\{p\}}_\bullet \shD_X$-filtration induced by this presentation. Since the rank of $\shE^f$ is 1, we want to show that $\gr^U(\shE^f)$ has rank zero, i.e.\ is torsion. To do this, it suffices to show that $\init_{(-1,1)}(b^2\partial_i - b^2\frac{\partial f}{\partial x_i}) = \init_{(-1,1)}(b^2\frac{\partial f}{\partial x_i})$ for some $i$. Equivalently, noticing that $\ord_{(-1,1)}(\partial f/\partial x_i)= \ord_{(-1,1)}(f) + 1$, we need to show that 
    \[ \ord_{(-1,1)}(f) > 0.\]
    But this is true exactly because $p$ is contained in the pole divisor of $f$.
\end{example}

\bigskip

\begin{example}\label{ex:gkz111012}
    Consider the $\shD_{\CC^3}$-module 
    \[ \shM = \frac{\shD_{\CC^3}}{\shD_{\CC^3}\{\partial_2^2 - \partial_1 \partial_3,\; E_1-\beta_1,\; E_2-\beta_2 \}},\]
    where $\beta_1,\beta_2\in \CC$, $E_1=x_1\partial_1 + x_2\partial_2+x_3\partial_3$, and $E_2 = x_2\partial_2 + 2x_3\partial_3$. This is the $A$-hypergeometric system corresponding to the matrix $A=\left[\begin{smallmatrix}
        1&1&1\\0&1&2
    \end{smallmatrix}\right]$
    and parameter $\beta=(\beta_1,\beta_2)$. Since the vector $(1,1,1)$ is in the row-span of $A$, $\shM$ is regular (\cite[Ch.~II, \S6.2, Th.]{hotta1998equivariant}). We show this using \Cref{th:Irr-is-0}.

    The rank of $\shM$ is 2. The characteristic variety of $\shM$ is
    \[ \Ch(\shM) = T^*_{\CC^3}\CC^3 \cup T^*_{V(x_1)}\CC^3 \cup T^*_{V(x_3)}\CC^3 \cup T^*_{V(x_2^2-4x_1x_3)}\CC^3.\]
    Therefore, we need to check that $\gr^U(\shM)$ has rank 2 when computed for general points of $V(x_1)$, $V(x_3)$, $V(x_2^2-4x_1x_3)$, along with a general point of the hyperplane at infinity. This is easily done using a computer algebra system such as Macaulay2 \cite{M2}. For instance, if $\beta = (1/4,1/4)$ and $p=(0,1,1)\in V(x_1)$ (which is general enough), we can compute $\gr^U(\shM)$ as follows: first translate the coordinate system to $p$ to get the $D_3$-ideal
    \[ D_3\{\partial_2^2-\partial_1\partial_3, x_1\partial_1 + (x_2-1)\partial_2 + (x_3-1)\partial_3 -\tfrac{1}{4}, (x_2-1)\partial_2 + 2(x_3-1)\partial_3 - \tfrac{1}{4}\}.\]
    Using Macualay2, one finds that this has initial ideal (with respect to the weight $(-1,-1,-1,1,1,1)$) 
    \[ D_3\{\partial_2, \partial_3, x_1\partial_1^2 + \partial_1\}.\]
    It is easy to see that this has rank 2; alternatively, the rank can be computed using Macaulay2.
\end{example}

\bigskip

\begin{example}\label{ex:gkz110011}
    Consider the $\shD_{\CC^3}$-module
    \[ \shM = \frac{\shD_{\CC^3}}{\shD_{\CC^3}\{\partial_2 - \partial_1 \partial_3,\; E_1-\beta_1,\; E_2-\beta_2 \}},\]
    where $\beta_1,\beta_2\in \CC$, $E_1=x_1\partial_1 + x_2\partial_2$, and $E_2 = x_2\partial_2 + x_3\partial_3$. This is the $A$-hypergeometric system corresponding to the matrix $A=\left[\begin{smallmatrix}
        1&1&0\\0&1&1
    \end{smallmatrix}\right]$
    and parameter $\beta=(\beta_1,\beta_2)$. Since the vector $(1,1,1)$ is \emph{not} in the row-span of $A$, $\shM$ is irregular (\cite[Cor.~3.16]{SW-irreg-gkz} or \cite[Th.~7.6]{nilssonAhyp}). We show this using \Cref{th:Irr-is-0}.
    
    The rank of $\shM$ is 2. The characteristic variety of $\shM$ is 
    \[\Ch(\shM) = T^*_{\CC^3}\CC^3\cup T^*_{V(x_1)}\CC^3 \cup T^*_{V(x_2)}\CC^3 \cup T^*_{V(x_3)}\CC^3 \cup T^*_{V(x_1,x_2)}\CC^3 \cup T^*_{V(x_2,x_3)}\CC^3.\]
    If we expected that $\shM$ were regular, then we would need to check that the rank of $\gr^U(\shM)$ equals 2 when computed at a general point of $V(x_1)$, $V(x_2)$, $V(x_3)$, and the hyperplane at infinity. However, we are confirming that $\shM$ is irregular. Following the same procedure as in \Cref{ex:gkz111012}, we find:
    \begin{itemize}
        \item At general points of $V(x_1)$, $V(x_3)$, and the hyperplane at infinity, the rank of $\gr^U(\shM)$ is 2, i.e.\ there is no irregularity.
        \item At a general point of $V(x_2)$, the rank of $\gr^U(\shM)$ is 1.
    \end{itemize}
\end{example}

\section{Fuchs' theorem for modules}\label{sec:Fuchs}
We are going to need a generalization of Fuchs' theorem (\cite[Th.~1.4.18]{sst}) which applies to arbitrary (holonomic) $D_1$-modules with arbitrary good $V_\bullet D_1$-filtrations rather than just cyclic ones. As part of the proof, we will need to know the equality in \Cref{lem:gr-and-rh}. Finally, we would like to thank Saiei-Jaeyong Matsubara-Heo for his help with the proofs in this section.

Set $\widehat{D}_1\coloneqq \CC\llbracket x\rrbracket\otimes_{\CC[x]} D_1$. Define the filtrations $V_\bullet D_1$ and $V_\bullet \widehat{D}_1$ by
\begin{align*}
    V_k D_1 &\coloneqq \Set{P\in D_1 | P(\CC[x] x^p) \subseteq \CC[x] x^{p-k}\text{ for all } p\in \ZZ}\\
    V_k \widehat{D}_1 &\coloneqq \Set{P\in D_1 | P(\CC\llbracket x\rrbracket x^p) \subseteq \CC\llbracket x\rrbracket x^{p-k}\text{ for all } p\in \ZZ}\\
        &= \CC\llbracket x\rrbracket\otimes_{\CC[x]} V_k D_1.
\end{align*}

When the good filtration $U_\bullet M$ is the canonical $V$-filtration (see \cite[Prop.~6.1.2]{sab93} for the definition), \Cref{lem:gr-and-rh} is well-known to be true, since in that case $\gr^V(M)$ is essentially isomorphic to $M$ \cite[Lem.~6.2.6]{sab93} (note that in this situation $M$ is assumed to be regular!). However, in the process of proving \Cref{th:Irr-is-0}, we are going to end up with a $D_1$-module with a good filtration that is not necessarily the canonical one.

\begin{lemma}\label{lem:gr-and-rh}
    Let $M$ be a regular holonomic $\widehat{D}_1$-module, and let $U_\bullet M$ be any $V_\bullet \widehat{D}_1$-good filtration. Then $\rank(\gr^U(M)) = \rank(M)$.
\end{lemma}
\begin{proof}
    We proceed via induction on the rank of $M$. 

    \smallskip
    \underline{Induction step.} Consider the map 
    \[ \varphi \colon M \longrightarrow M[x^{-1}].\]
    Since $M[x^{-1}]$ is a regular meromorphic connection, it has a rank 1 submodule $\widetilde{M}'$. Set $M'=\varphi^{-1}(\widetilde{M}')$ and $M''=M/M'$. Then we have a short exact sequence
    \[ 0 \to M' \to M \to M''\to 0.\]
    Equip $M'$ and $M''$ with the induced filtrations. Then $\gr^U$ preserves the exactness of the sequence, so
    \begin{align*}
        \rank(\gr^U M)
        &= \rank(\gr^U M') + \rank(\gr^U M'')\\
        &= \rank(M') + \rank(M'') \qquad (\text{by the induction hypothesis})\\
        &= \rank(M).
    \end{align*}
    
    \smallskip
    \underline{Base case.} Assume that $\rank(M)=1$.

    \smallskip
    \emph{Torsion case.} Assume that $M$ is torsion. Then $\gr^U(M)$ is also torsion, so both $M$ and $\gr^U(M)$ have rank zero.

    \smallskip
    \emph{Meromorphic case.} Assume that $M$ is meromorphic. Then up to isomorphism, $M=\CC\llbracket x\rrbracket[x^{-1}]x^\alpha$ for some $\alpha\in \CC$. Since $U_\bullet M$ is a good filtration, there exists $m_1,\ldots,m_s\in M$ and $p_1,\ldots,p_s\in \ZZ$ such that 
    \[ U_\bullet M = \sum_{i=1}^s V_{\bullet-p_i}\widehat{D}_1\cdot m_i.\]
    Since $\CC\llbracket x\rrbracket$ is a DVR with uniformizer $x$, and because $\CC\llbracket x\rrbracket\subseteq V_0\widehat{D}_1$ and $U_\bullet M$ is a $V_\bullet\widehat{D}_1$-filtration, we may assume that $m_i=x^{\alpha+k_i}$ for some $k_i$. Then
    \[ V_{\bullet-p_i}\widehat{D}_1\cdot m_i = V_{\bullet-p_i-k_i}\widehat{D}_1\cdot x^\alpha.\]
    Without loss of generality, assume that $p_1+k_1$ is the largest of the $(p_i+k_i)$'s. Then
    \[ U_\bullet M = V_{\bullet-p_1-k_1}\widehat{D}_1 x^\alpha.\]
    Then, up to twist in the grading,
    \[ \gr^U M = \CC[x][x^{-1}]x^{\alpha},\]
    which also has rank 1.

    \smallskip
    \emph{Torsion-free case.}
    Equip $M[x^{-1}]$ and $M[x^{-1}]/M$ with the induced filtrations. Then $\gr^U$ preserves the exactness of the short exact sequence (which is exact by torsion-freeness)
    \[ 0 \to M \to M[x^{-1}] \to M[x^{-1}]/M \to 0.\]
    Now use that the middle module is meromorphic and the right module is torsion, and that $\rank$ is additive on short exact sequences.
    
    \smallskip
    \emph{General case.}
    There is a short exact sequence
    \[ 0 \to K \to M \to C \to 0,\]
    where $K$ is torsion and $C=\im(M\to M[x^{-1}])$ is torsion-free. Equip $K$ and $C$ with the induced filtrations. Then applying $\gr^U$ preserves exactness, and we can use the previous cases to conclude.
\end{proof}

\begin{theorem}\label{th:Fuchs-for-mods}
    Let $M$ be a holonomic $D_1$-module, and let $U_\bullet M$ be any $V_\bullet D_1$-good filtration. Then $M$ is regular at 0 if and only if $\rank(\gr^U(M)) = \rank(M)$. 
\end{theorem}
\begin{proof}
    Let $\widehat{M}=\CC\llbracket x\rrbracket\otimes_{\CC[x]} M$. This has the same rank as $M$, and it is standard that $\widehat{M}$ is regular at 0 (as a $\widehat{D}_1\coloneqq \CC\llbracket x\rrbracket\otimes_{\CC[x]} D_1$-module) if and only if the same is true for $M$ (as a $D_1$-module). The filtration $U_\bullet M$ induces a filtration $U_\bullet \widehat{M}$ on $\widehat{M}$ which is $V_\bullet \widehat{D}_1$-good. 
    
    According to \cite[Th.~6.3.1]{sab93} and its proof, $\widehat{M}$ decomposes as $\widehat{M}=\widehat{M}_r \oplus \widehat{M}_i$, where $\widehat{M}_r$ is regular, $\widehat{M}_i$ is an irregular meromorphic connection, and 
    \[ \widehat{M}_i = \bigcap_{\lambda \in \ZZ} U_\lambda \widehat{M}.\]
    So, all irregularity disappears upon passing to $\gr^U \widehat{M} (= \gr^U M)$, i.e.\ 
    \[ \gr^U(\widehat{M}_r) = \gr^U(\widehat{M}) (= \gr^U(M)).\]
    Thus, 
    \[ \rank(\gr^U(M)) = \rank(\gr^U \widehat{M}_r) = \rank(\widehat{M}_r),\]
    where the second equality is  by \Cref{lem:gr-and-rh}.
    Thus,
    \[ \rank(M) = \rank(\widehat{M}) = \rank(\widehat{M}_r) + \rank(\widehat{M}_i) = \rank(\gr^U(M)) + \rank(\widehat{M}_i).\]
    Now use that since $\widehat{M}_i$ is meromorphic, its rank vanishes if and only if $\widehat{M}_i$ is zero, i.e.\ if and only if $M$ is regular.
\end{proof}

\section{Restriction and associated graded modules}\label{sec: restr-and-gr}
In general, there is no reason to expect that restriction to a smooth subvariety commutes with taking $\gr^U$, i.e.\ that the surjection in \Cref{gr-surj} below is an isomorphism. In this section, we show that this commutation does work as long as the restriction is non-characteristic (\Cref{hypsf-non-char-gr} in the hypersurface case, \Cref{th:non-char-gr} in the general case). Finally, we would like to thank Henry Dakin for his help with the proof of \Cref{hypsf-non-char-gr}.

Given a $D_n$-module $M$ equipped with a $V_\bullet D_n$-filtration with respect to a weight $w\in \RR^n$, we equip the pullback $i^*M$ via the map $i\colon \CC^r\hookrightarrow \CC^n$, $(x_1,\ldots,x_r)\mapsto (x_1,\ldots,x_r,0,\ldots,0)$, with the following filtration:
\begin{equation}
    U_\bullet i^*M \coloneqq \frac{U_\bullet M}{\left(\sum_{j=r+1}^n x_jM\right)\cap U_\bullet M}.
\end{equation}
This is a $V_\bullet D_r$-filtration with respect to the weight $(w_1,\ldots,w_r)$. Here, we have used that there is a canonical isomorphism
\[i^*M\cong \frac{M}{\sum_{j=r+1}^n x_jM}\]
of $D_r$-modules (see, e.g., the discussion of restriction in \cite[\S5.2]{sst} or \cite[\S1.3]{htt}).

\begin{lemma}\label{gr-surj}
    Let $M$ be a finitely-generated $D_n$-module equipped with a good $V_\bullet D_n$-filtration with respect to a weight $w\in \RR^n$. Let $i\colon \CC^r\hookrightarrow \CC^n$ be given by $(x_1,\ldots,x_r)\mapsto (x_1,\ldots,x_r,0,\ldots,0)$. Then there is a natural surjection
    \[ \gr^U(i^*M)\twoheadrightarrow i^*\gr^U(M).\]
\end{lemma}
\begin{proof}
    We have
    \begin{align*}
        \gr^U_k(i^*M)
        &= \frac{U_k M}{U_{k-1}M + \left(\sum_{j=r+1}^n x_jM\right)\cap U_k M}
        \intertext{and}
        (i^*\gr^U(M))_k
        &= \frac{U_kM}{U_{k-1}M + \sum_{j=r+1}^n x_jU_kM}.
    \end{align*}
    From this, it is immediate that there is a surjective map of graded vector spaces $\gr^U(i^*M)\twoheadrightarrow i^*\gr^U(M)$, and it is easy to see that this map is $D_r$-linear.
\end{proof}

\subsection{Smooth hypersurfaces}
In this subsection, we first show (\Cref{hypsf-nc-implies-gr-nc}) in the hypersurface case that a non-characteristic restriction for $\shM$ is also non-characteristic for $\gr^U(\shM)$. This is then used in various places in the proof of \Cref{hypsf-non-char-gr}.

\begin{lemma}\label{hypsf-nc-implies-gr-nc}
    Let $M$ be a coherent $D_n$-module equipped with a good $V_\bullet D_n$-filtration (with respect to a weight $w\in \RR^n$ with $0<w_i\leq w_n$ for all $i$). Let $i\colon \CC^{n-1}\hookrightarrow \CC^n$ be $(x_1,\ldots,x_{n-1})\mapsto (x_1,\ldots,x_{n-1},0)$. If $i$ is non-characteristic for $M$ in a neighborhood of the origin, then the same is true for $\gr^U(M)$.
\end{lemma}
\begin{proof}
    Since $i$ is non-characteristic for $M$ and is the inclusion of a hypersurface, there exists \cite[Lem.~2.4.7]{htt} a surjection 
    \[ \bigoplus_j \frac{D_n}{D_n P_j} \twoheadrightarrow M, \]
    where $i$ is non-characteristic with respect to each $P_j$. We can (and do) arrange for this to be a strict filtered map such that each $D_n/D_nP_j$ has a filtration of the form
    \begin{equation*}
        U_\bullet(D_n/D_nP_j) = \frac{V_{\bullet-k_j}D_n}{D_nP_j\cap V_{\bullet-k_j}D_n}.
    \end{equation*}
    Applying $\gr^U$ to this surjection gives (up to shift in the grading) a surjection
    \begin{equation}
        \bigoplus_j \frac{D_n}{D_n\init_{(-w,w)}(P_j)} \twoheadrightarrow \gr^U(M).
    \end{equation}
    Since this is a surjection, the characteristic variety of the codomain is contained in the characteristic variety of the domain. Therefore, to prove that $i$ is non-characteristic for $\gr^U(M)$, it is enough to prove that $i$ is non-characteristic for each $D_n/D_n\init_{(-w,w)}(P_j)$.

    Let $m=\ord(P_j)$. The morphism $i$ being non-characteristic for $P_j$ in a neighborhood of the origin is equivalent to (see \cite[Ex.~2.4.4]{htt}) $\partial_n^m$ being a monomial of $P_j$. So, to show that $i$ is non-characteristic for $\init_{(-w,w)}(P_j)$  in a neighorhood of the origin, it suffices to show that $\partial_n^m$ is also a monomial of $\init_{(-w,w)}(P_j)$.
    
    Let $\ell$ be the $(-w,w)$-weight of $\init_{(-w,w)}(P_j)$. If $x^a\partial^b$ is a monomial of $P_j$ with weight $\ell$, then
    \begin{align*}
        w_n m &\leq \ell\\
            &= w_1b_1 + \cdots w_nb_n - w_1a_1- \cdots - w_na_n\\
            &\leq w_1b_1 + \cdots w_nb_n\\
            &\leq w_n(b_1+\cdots + b_n)\\
            &\leq w_nm,
    \end{align*}
    where in the second-to-last inequality we used that $w_i\leq w_n$ for all $i$. Thus, $\ell=w_nm$, which means that, indeed, $\partial_n^m$ is a monomial of $\init_{(-w,w)}(P_j)$.
\end{proof}


    

\begin{proposition}\label{hypsf-non-char-gr}
    Let $M$ be a coherent $D_n$-module equipped with a good $V_\bullet D_n$-filtration (with respect to a weight $w\in \RR^n$ with $0<w_i\leq w_n$ for all $i$). Let $i\colon \CC^{n-1}\hookrightarrow \CC^n$ be $(x_1,\ldots,x_{n-1})\mapsto (x_1,\ldots,x_{n-1},0)$. If $i$ is non-characteristic for $M$ (in a neighborhood of the origin), then the natural map
    \[\gr^U(i^*M)\twoheadrightarrow i^*\gr^U(M)\]
    is an isomorphism (in a neighborhood of the origin).
\end{proposition}
\begin{proof}
    Since $i$ is non-characteristic for $M$ and is the inclusion of a hypersurface, there exists \cite[Lem.~2.4.7]{htt} an exact sequence
    \begin{equation}\label{eq:dim2Pseq}
        0 \to K \to \bigoplus_j \frac{D_n}{D_n P_j} \to M \to 0,
    \end{equation}
    where $i$ is non-characteristic with respect to each $P_j$. We can (and do) arrange for this to be a strict filtered exact sequence such that each $D_n/D_nP_j$ has a filtration of the form
    \begin{equation}\label{eq:dim2DDPfil}
        U_\bullet(D_n/D_nP_j) = \frac{V_{\bullet-k_j}D_n}{D_nP_j\cap V_{\bullet-k_j}D_n}
    \end{equation}
    for some $k_j\in \ZZ$. By non-characteristicness, the sequence
    \begin{equation}
        0\to i^*K \to i^*\bigoplus_j \frac{D_n}{D_n P_j} \xrightarrow{f} i^*M \to 0
    \end{equation}
    is exact (see \cite[Th.~2.4.6(i)]{htt}), and an easy argument shows that $f$ is strict. Therefore, $\gr^U(f)$ is onto. However, there doesn't appear to be a reason why the inclusion of $i^*K$ into $i^*\bigoplus_j \frac{D_n}{D_n P_j}$ should be strict. Therefore, we replace $i^*K$ with the filtered kernel of $f$, which we'll denote $\ker(f)$; this replacement will be accounted for in Claim~\ref{claim:dash-exists} below. Note that $\ker(f)$ is just $i^*K$ equipped with the filtration induced by the inclusion into $i^*\bigoplus_j \frac{D_n}{D_n P_j}$. Then we get an exact sequence
    \begin{equation}
        0 \to \gr^U(\ker(f)) \to \gr^U\left(i^*\bigoplus_j \frac{D_n}{D_n P_j}\right) \xrightarrow{\gr^U(f)} \gr^U(i^*M) \to 0.
    \end{equation}

    On the other hand, \eqref{eq:dim2Pseq} is strict, and by \Cref{hypsf-nc-implies-gr-nc}, $i$ is non-characteristic for $\gr^U(M)$. So, we get an exact sequence
    \begin{equation}
        0\to i^*\gr^U(K) \to i^*\gr^U\left(\bigoplus_j \frac{D_n}{D_n P_j}\right) \xrightarrow{f} i^*\gr^U(M) \to 0.
    \end{equation}

    The remainder of the proof will involve studying the following diagram:
    \begin{equation}\label{eq:dim2gr-diagram}
    \begin{tikzcd}
    	0 & {\gr^U(\ker(f))} & {\gr^U\left(i^*\bigoplus_j \frac{D_n}{D_n P_j}\right)} & {\gr^U(i^*M)} & 0 \\
    	0 & {i^*\gr^U(K)} & {i^*\gr^U\left(\bigoplus_j \frac{D_n}{D_n P_j}\right)} & {i^*\gr^U(M) } & {0.}
    	\arrow[from=1-1, to=1-2]
    	\arrow[from=1-2, to=1-3]
    	\arrow["{\gr^U(f)}", from=1-3, to=1-4]
    	\arrow[from=1-4, to=1-5]
    	\arrow[from=2-1, to=2-2]
    	\arrow[from=2-2, to=2-3]
    	\arrow["g", from=2-3, to=2-4]
    	\arrow[from=2-4, to=2-5]
    	\arrow[dashed, two heads, from=1-2, to=2-2]
    	\arrow["\alpha", two heads, from=1-3, to=2-3]
    	\arrow[two heads, from=1-4, to=2-4]
    \end{tikzcd}
    \end{equation}
    Note that we already know that the right-hand square of the diagram commutes, and that both rows are exact. 
    
    \medskip
    \noindent\underline{Claim \refstepcounter{claim}\arabic{claim}.}\label{claim:middle-isom} The surjection $\alpha$ in \eqref{eq:dim2gr-diagram} is an isomorphism.

    \smallskip
    \noindent\emph{Proof of the claim.}
    Since $\gr^U$ and $i^*$ both commute with direct sums, it suffice to prove the case where there is a single $P$, i.e.: If $i$ is non-characteristic for the element $P\in D_n$, then the surjection 
    \[ \alpha\colon \gr^U(i^*(D_n/D_nP)) \twoheadrightarrow i^*\gr^U(D_n/D_nP)\]
    is an isomorphism. We recall from \eqref{eq:dim2DDPfil} that the filtration on $D_n/D_nP$ is of the form
    \[ U_\bullet(D_n/D_nP) = \frac{V_{\bullet-k}D_n}{D_nP\cap V_{\bullet-k}D_n}\]
    for some $k\in \ZZ$. Without loss of generality, we may assume that $k=0$.

    To begin with, notice that $\gr^U(D_n/D_nP) = D_n/D_n\init_{(-w,w)}(P)$. By assumption, $i$ is non-characteristic for $P$ and (by \Cref{hypsf-nc-implies-gr-nc}) for $\init_{(-w,w)}(P)$. Then by \cite[Ex.~2.4.4]{htt}, 
    \begin{align*}
        \gr^U(i^*(D_n/D_nP)) &\cong \gr^U(D_{n-1}^{\oplus m}) \cong D_{n-1}^{\oplus m}
        \intertext{and}
        i^*\gr^U(D_n/D_nP) &= i^*(D_n/D_n\init_{(-w,w)}(P))\cong D_{n-1}^{\oplus m'},
    \end{align*}
    where $m=\ord(P)$ and $m'=\ord(\init_{(-w,w)}(P))$. So, $\alpha$ is a surjective map from $D_{n-1}^{\oplus m}$ to $D_{n-1}^{\oplus m'}$. But the Weyl algebra is in particular Noetherian, so a surjective \emph{endomorphism} of a finitely-generated $D_{n-1}$-module is necessarily an isomorphism\footnote{This is easy to prove using Noetherian-ness. One applies Noetherian-ness to the sequence $\ker(\alpha)\subseteq \ker(\alpha^2)\subseteq \ker(\alpha^3)\subseteq\cdots$, where $\alpha^k$ is the $k$th iterate of $\alpha$. Surjectivity and the Snake Lemma then imply that $\alpha$ is an isomorphism.}. Thus, it remains to show that $m=m'$, so that $\alpha$ is indeed a (surjective) endomorphism.

    The morphism $i$ being non-characteristic for $P$ implies in particular (see \cite[Ex.~2.4.4]{htt}) that $P$ 
    has a monomial of the form $\partial_n^m$. 
    The argument in the proof of \Cref{hypsf-nc-implies-gr-nc} implies that $\partial_n^m$ is also a monomial of $\init_{(-w,w)}(P)$. So, $m'=m$.
    \hfill/\!/\!/
    
    \medskip
    \noindent\underline{Claim \refstepcounter{claim}\label{claim:dash-exists}\arabic{claim}.} The dashed surjection in \eqref{eq:dim2gr-diagram} exists and makes the entire diagram commute.

    \smallskip
    \noindent\emph{Proof of the claim.}
    By the universal property of kernels, we get a filtered map $\varphi\colon i^*K\to \ker(f)$. This induces a graded map $\gr^U(\varphi)\colon \gr^U(i^*K)\to \gr^U(\ker(f))$. Moreover, $\alpha(\gr^U(\ker(f)))$ is in the kernel of $g$, so the universal property of kernels gives the dashed map, which we'll call $\psi$. On the other hand, we also know that the following diagram commutes:
    \[\begin{tikzcd}
    	{\gr^U(i^*K)} & {\gr^U(\ker(f))} & {\gr^U\left(i^*\bigoplus_j \frac{D_n}{D_n P_j}\right)} \\
    	{i^*\gr^U(K)=\ker(g)} && {i^*\gr^U\left(\bigoplus_j \frac{D_n}{D_n P_j}\right).}
    	\arrow[hook, from=2-1, to=2-3]
    	\arrow[two heads, from=1-1, to=2-1]
    	\arrow["{\gr^U(\varphi)}", from=1-1, to=1-2]
    	\arrow[hook, from=1-2, to=1-3]
    	\arrow["\alpha", from=1-3, to=2-3]
    \end{tikzcd}\]
    So, $\alpha\circ \gr^U(\varphi)$ maps $\gr^U(i^*K)$ surjectively onto $i^*\gr^U(K)$, which implies that $\alpha$ maps $\gr^U(\ker(f))$ surjectively onto $i^*\gr^U(K)$. Thus, since $\alpha$ restricted to $\gr^U(\ker(f))$ is just $\psi$, we see that $\psi$ is surjective. \hfill/\!/\!/

    \medskip
    Now apply the snake lemma along with the two claims to conclude.
\end{proof}

\subsection{Smooth subvarieties}
In this section we use \Cref{nonchar-flag} to extend the results of the previous section to smooth subvarieties of arbitrary codimension.

\begin{lemma}\label{nonchar-flag}
    Let $Z\subseteq X$ be a smooth subvariety of dimension $r$, $\shM$ a coherent $\shD_X$-module. If $Z$ is non-characteristic for $\shM$ near $p$, then there exists a sequence
    \[ Z=Z_r\subsetneq \cdots Z_n \subsetneq X\]
    of smooth subvarieties such that for all $k$, $\dim Z_k=k$ and $Z_k$ is non-characteristic for $\shM$ near $p$.
\end{lemma}
\begin{proof}
    Let $x_1,\ldots,x_n$ be coordinates centered at $p$ such that $Z$ is cut out by $x_{r+1},\ldots,x_n$. Set 
    \[Z_k=V(x_{k+1},\ldots,x_n)\] 
    for $k=r,\ldots,n-1$. By definition of non-characteristic, and because non-characteristcness is an open property, $Z_k$ is non-characteristic for $\shM$ near $p$ if and only if
    \[(p;0,\ldots,0,\xi_{k+1},\ldots,\xi_n)\in \Ch(\shM)\implies \xi_{k+1}=\cdots=\xi_n=0.\]
    Therefore, if $Z$ is non-characteristic for $\shM$ near $p$, the same is true of each $Z_k$.
\end{proof}

\begin{lemma}\label{nc-implies-gr-nc}
    Let $M$ be a coherent $D_n$-module equipped with a good $V_\bullet D_n$-filtration (with respect to a weight $w\in \RR^n$ with $0<w_i\leq w_n$ for all $i$). Let $i\colon \CC^r\hookrightarrow \CC^n$ be $(x_1,\ldots,x_{r})\mapsto (x_1,\ldots,x_r,0)$. If $i$ is non-characteristic for $M$ in a neighborhood of the origin, then the same is true for $\gr^U(M)$.
\end{lemma}
\begin{proof}
    This follows immediately from \Cref{hypsf-nc-implies-gr-nc} and \Cref{nonchar-flag}.
\end{proof}

\begin{theorem}\label{th:non-char-gr}
    Let $M$ be a coherent $D_n$-module equipped with a good $V_\bullet D_n$-filtration (with respect to a weight $w\in \RR^n$ with $w_1=\cdots=w_n>0$). Let $i\colon \CC^r\hookrightarrow \CC^n$ be $(x_1,\ldots,x_r)\mapsto (x_1,\ldots,x_r,0,\ldots,0)$. If $i$ is non-characteristic for $M$ (in a neighborhood of the origin), then the natural map
    \[\gr^U(i^*M)\twoheadrightarrow i^*\gr^U(M)\]
    is an isomorphism (in a neighborhood of the origin).
\end{theorem}
\begin{proof}
    This follows immediately from \Cref{hypsf-non-char-gr} and \Cref{nonchar-flag}.
\end{proof}

\section{Proof of \Cref{th:Irr-is-0}}\label{sec:proof}
We will need one more technical \namecref{Irr-vs-restr} before proceeding with the proof of \Cref{th:Irr-is-0}. The proof of \ref{Irr-vs-restr.irr} requires derived category techniques from the theory of $D$-modules, but readers unfamiliar with these techniques may safely skip the proof.

\begin{lemma}\label{Irr-vs-restr}
    Let $\shM$ be a holonomic $\shD_X$-module, $Y$ a hypersurface in $X$ which contains $\Sing(\shM)$. Let $\{S_j\}$ be a stratification of $Y$ such that $\Ch(\shM)\cup \Ch(\shM(*Y))$ is contained in $T^*_XX\cup \bigcup_j \overline{T^*_{S_j}X}$. Choose $j$ such that $\dim S_j=\dim Y$, and let $p\in S_j$. If $i\colon C\hookrightarrow X$ is a smooth curve transverse to $S_j$ at $p$, then 
    \begin{enumerate}[label=\textnormal{(\alph*)}]
        \item\label{Irr-vs-restr.nc} $i$ is non-characteristic for $\shM$ and $\shM(*Y)$ near $p$, and 
        \item\label{Irr-vs-restr.irr} $\Irr_{\{p\}}(i^*\shM^\an) \cong \Irr_Y(\shM^\an)_p$.
    \end{enumerate}
\end{lemma}
\begin{proof}
    \ref{Irr-vs-restr.nc} We only prove the non-characteristicness of $i$ with respect to $\shM$. The same argument works for $\shM(*Y)$. Since $C$ is transverse to $S_j$ at $p$, there are local coordinates $x_1,\ldots,x_n$ centered at $p$ such that $C$ is cut out by $x_2,\ldots,x_n$ and $S_j$ is cut out by $x_1$. Then, letting $x_1,\ldots,x_n,\xi_1,\ldots,\xi_n$ be the induced coordinates on $T^*X$ near $p$, we have
    \[ (T^*_{S_j}X)_p = V(\xi_2,\ldots,\xi_n)\]
    and
    \[ (T^*_CX)_p = V(\xi_1),\]
    which clearly intersect only at $0\in T^*_pX$. Hence, noticing that $\Ch(\shM)\cap T^*_pX = (T^*_{S_i}X)_p$ by hypothesis, we see that $i$ is non-characteristic for $\shM$ near $p$. 

    \smallskip
    \ref{Irr-vs-restr.irr} Recall from \eqref{eq:alt-Irr} that, up to cohomological shift,
    \[ \Irr_Y(\shM^\an) \cong k^{-1}\shRHom_{\shD_{X^\an}}(\shM^\an(*Y), \shO_{X^\an}),\]
    where $k\colon Y\hookrightarrow X$ is inclusion. Also,
    \[\Irr_{\{p\}}(i^*\shM^\an) \cong \shRHom_{\shD_{C^\an}}((i^*\shM^\an)(*p), \shO_{C^\an})_p.\]
    By \cite[Prop.~3.5-5]{Meb04}, $(i^*\shM^\an)(*p) \cong i^*(\shM^\an(*Y))$ in a neighborhood of $p$. Now use \ref{Irr-vs-restr.nc} and the Cauchy--Kowalevski--Kashiwara theorem (\cite[Th.~4.3.2]{htt}) to conclude.
\end{proof}

By \cite[Th.~3.1.1]{Meb89}, the support of $\Irr_Y(\shM^\an)$ is a union of irreducible components of $Y$. Therefore, $\Irr_Y(\shM^\an)$ vanishes if and only if it vanishes near some/every point $p$ of each $\dim{Y}$-dimensional stratum $S_j$. So, let $S_j$ be a $\dim{Y}$-dimensional stratum $S_j$, and let $p\in S_j$. Choose a good $V_\bullet^{\{p\}}\shD_X$-filtration on $\shM$ near $p$. We need to show that $\rank(\gr^U(\shM)) = \rank(\shM)$ if and only if $\Irr_Y(\shM^\an)_p=0$.

Choose a smooth curve $i\colon C\hookrightarrow X$ transverse to $S_j$ at $p$. By \Cref{Irr-vs-restr} and \Cref{nc-implies-gr-nc}, $i$ is non-characteristic for both $\shM$ and $\gr^U(\shM)$ at $p$, so by the Cauchy--Kowalevski--Kashiwara theorem (\cite[Th.~4.3.2]{htt}), 
\begin{equation}\label{eq:rank-and-restr+gr}
    \rank(\shM) = \rank(i^*\shM) \qquad \text{and}\qquad \rank(\gr^U(\shM)) = \rank(i^*\gr^U(\shM)).
\end{equation}
By \Cref{th:non-char-gr} applied to $\shM$, $i^*$ commutes with $\gr^U$, so 
\begin{equation}\label{eq:rank,restr,gr}
    \rank(i^*\gr^U(\shM)) = \rank(\gr^U(i^*\shM)).
\end{equation}
Finally, \Cref{Irr-vs-restr} implies that $i^*\shM$ is regular at $p$ if and only if $\shM$ is regular at $p$; so, by Fuchs' theorem for modules (\Cref{th:Fuchs-for-mods}), $\rank(i^*\shM) = \rank(\gr^U(i^*\shM))$ if and only if $\shM$ is regular at $p$. Combining this with \eqref{eq:rank-and-restr+gr} and \eqref{eq:rank,restr,gr} proves the result.

\section{Nilsson solutions}\label{sec:Nilsson}
Recall from the introduction (or \cite[Th.~1.4.18]{sst}) that Fuchs' theorem says that in one dimension, regularity of $P$ is equivalent to the ability to write every multivalued solution as a linear combination of functions of the form $x^\lambda g(x)(\log x)^k$, where $\lambda \in \CC$ and $g(x)$ is holomorphic. Such functions are called Nilsson functions. There is an analogous notion of Nilsson functions (or rather Nilsson series) in higher dimensions. We now show a higher-dimensional analog of Fuchs' theorem which uses these higher-dimensional Nilsson series. 

The following definitions follow those in \cite[\S2]{nilssonAhyp}.

\begin{definition}\label{def:weight-vec}
    Let $I$ be a holonomic $D_n$-ideal. A vector $w\in \RR^n$ is called a \emph{generic weight vector} for the $I$ if there exists a strongly convex open cone $\calC\subseteq \RR^n$ containing $w$ such that for all $w'\in \calC$,
    \[ \init_{(-w',w')}(I) = \init_{(-w,w)}(I).\]
\end{definition}

Denote the dual cone of $\calC$ by $\calC^*$. 

\begin{definition}
    Let $I$ be a holonomic $D_n$-ideal. Let $w$ be a generic weight vector for $I$. A formal solution $\varphi$ of $I$ is called a \emph{basic Nilsson solution of $I$ in the direction of $w$} if it has the form
    \begin{equation}
        \varphi = \sum_{u\in C}x^{v+u}p_u(\log(x_1),\ldots,\log(x_n)),
    \end{equation}
    for some vector $v\in \CC^n$ such that
    \begin{enumerate}
        \item $C$ is contained in $\calC^*\cap \ZZ^n$, where $\calC$ is as in \Cref{def:weight-vec},
        \item the $p_u$ are polynomials, and there exists a $K\in \ZZ$ such that $\deg(p_u)<K$ for all $u\in C$,
        \item $p_0\neq 0$.
    \end{enumerate}
    The $\CC$-span of the basic Nilsson solutions of $I$ in the direction of $w$ is called the \emph{space of formal Nilsson solutions of $I$ in the direction of $w$} and is denoted $\calN_w(I)$. 
\end{definition}

\begin{theorem}\label{th:nilsson}
    Let $X$ be a smooth \emph{complete} variety. Let $\mathcal{M}$ be a meromorphic connection on $X$. Then $\shM$ is regular if and only if for all $p\in X$, there exists
    \begin{itemize}
        \item local coordinates $x_1,\ldots,x_n$ centered at $p$, and
        \item an ideal $I$ of $D_n$ such that $\mathcal{M}\cong \mathcal{D}_X/\mathcal{D}_X I$ on a neighborhood of $p$
    \end{itemize}
    such that for all generic weight vectors $w\in \RR^n$ of $I$, 
    \[ \dim_\CC \calN_w(I) = \rank(\shM).\]
\end{theorem}
\begin{proof}
    $(\Rightarrow)$ This is the content of \cite[\S2.5]{sst}. 
    
    \smallskip
    $(\Leftarrow)$ By \cite[Th.~2.5.5 and Prop.~2.5.7]{sst}, the rank of $\init_{(-w,w)}(I)$ is at least the dimension of $\calN_w(I)$, which by hypothesis equals the rank of $\shM$. Now apply \Cref{th:const-implies-rh} and the semi-continuity of rank \cite[Th.~2.2.1]{sst}.
\end{proof}

\begin{remark}
    Notice that \Cref{th:nilsson} says nothing about the convergence of the formal Nilsson solutions. However, \cite[Cor~2.4.16]{sst} implies that under the hypotheses of \Cref{th:nilsson}, every formal Nilsson solution is in fact convergent.
\end{remark}

\section{Regularity algorithm}\label{sec:algorithm}
We now rewrite \Cref{th:Irr-is-0} in the form of an algorithm (\Cref{alg:comps-of-irregularity}) to compute the support of the irregularity complex. An immediate consequence of this algorithm is an algorithm (\Cref{alg:irregularity}) to decide whether a meromorphic connection on a smooth complete variety is regular.

\begin{algorithm}[H]
    \caption{The support of the irregularity complex}
    \label{alg:comps-of-irregularity}
    \DontPrintSemicolon
    \LinesNotNumbered
    \KwIn{A holonomic $D$-module $\shM$ on a variety $X$, and a hypersurface $Y$ containing $\Sing\shM$}
    \KwOut{The set of irreducible components of the support of $\Irr_Y(\shM)$}
    \nl Compute $\rank(\shM)$.\;
    \nl Compute the irreducible components $\Lambda_1,\ldots,\Lambda_s$ of $\Ch(\shM)\cup \Ch(\shM(*Y))$ other than $T^*_XX$.\;
    \nl Let $Y_j$ be the projection of $\Lambda_j$ onto $X$.\;
    \nl \For{$j$ such that $\dim Y_j = \dim X - 1$}{
        \nl Choose $p \in Y_j\setminus \bigcup_{k\neq j} Y_k$.\label{alg:comps-of-irregularity.p}\;
        \nl Choose a presentation $\shM\cong \shD_X^{\oplus \ell}/\shN$ near $p$, and equip $\shM$ with the induced $V_\bullet^{\{p\}}\shD_X$-filtration.\;
        \nl Compute $\rank(\gr^U(\shM))$.\;
        \nl \If{$\rank(\gr^U(\shM)) \neq \rank(\shM)$}{
            \nl Append $Y_j$ to $\mathcal{C}$.\;
        }
    }
    \nl \Return{$\calC$}
\end{algorithm}

\begin{remark}
    When choosing the point $p$ in \Cref{alg:comps-of-irregularity}, it is not sufficient to only look at the complement in $Y_j$ of those $Y_k$ having codimension 1.
\end{remark}

\begin{remark}
    Step~\ref{alg:comps-of-irregularity.p} of \Cref{alg:comps-of-irregularity} can be accomplished symbolically using standard techniques: First, find hypersurfaces cutting out all the $Y_k$ for $j\neq k$, and use these to exhibit $Y_j\setminus\bigcup_{k\neq j} Y_k$ as an affine variety. Next, find a point in this affine variety via the \texttt{independentSets} command in \cite{M2}, computing a minimal prime, then extending the base field so that the computed minimal prime splits as an intersection of linear maximal ideals.
\end{remark}

\begin{question}
    Since the rank of $\gr^U(\shM)$ depends only on the irreducible component $Y_j$, the following question is therefore natural: Is there a way to compute this quantity ``directly'' from $\shM$ without requiring one to find a $p$?
\end{question}

\begin{algorithm}[H]
    \caption{Decide whether a meromorphic connection is regular}
    \label{alg:irregularity}
    \DontPrintSemicolon
    \LinesNotNumbered
    \KwIn{A meromorphic connection $\shM$ on a complete variety $X$}
    \KwOut{Whether or not $\shM$ is regular}
    \nl Compute $Y=\Sing(\shM)$.\;
    \nl Compute the support of $\Irr_Y(\shM)$ using \Cref{alg:comps-of-irregularity}.\;
    \nl \If{$\Supp(\Irr_Y(\shM)) = \emptyset$}{
        \Return{True}\;
    }
    \nl \Else{\Return{False}\;}
\end{algorithm}

\section{A divisor measuring irregularity}\label{sec:irrdiv}
The numbers in \Cref{th:Irr-is-0}, applied to meromorphic connections, can be collected into a divisor, which we are calling the \emph{irregularity divisor}.

\begin{theorem}\label{th:irrdiv-existence}
    Let $X$ be a smooth variety, and let $\shM$ be a meromorphic connection on $X$ with pole divisor $Y$. Let $\{Y_i\}$ be the irreducible components of $Y$, and let $p_i$ be a general point of $Y_i$. Let $U_{i,\bullet}\shM$ be a good $V^{\{p_i\}}_\bullet\shD_X$-filtration, and set
    \begin{equation}
        \irrdiv(\shM) \coloneqq \sum_{i} \irrmult(Y_i,\shM)\cdot Y_i \in \Div(X),
    \end{equation}
    where
    \begin{equation}
        \irrmult(Y_i,\shM)\coloneqq \rank(\shM) - \rank(\gr^U\shM).
    \end{equation}
    Then 
    \begin{enumerate}[label=\textnormal{(\alph*)}]
        \item\label{th:irrdiv-existence.indep} $\irrdiv(\shM)$ is independent of the choice of $p_i$ and $U_{i,\bullet}\shM$,
        \item\label{th:irrdiv-existence.eff} $\irrdiv(\shM)$ is effective,
        \item\label{th:irrdiv-existence.supp} $\Supp(\irrdiv(\shM)) = \Supp(\Irr_Y(\shM^\an))$,
        \item\label{th:irrdiv-existence.an} $\shM^\an$ is regular if and only if $\irrdiv(\shM)=0$, and
        \item\label{th:irrdiv-existence.compl} if $X$ is complete, then $\shM$ is regular if and only if $\irrdiv(\shM) = 0$.
    \end{enumerate}
\end{theorem}
\begin{proof}
    \ref{th:irrdiv-existence.indep} This follows from the proof of \Cref{th:Irr-is-0}, noticing that the statement is true in dimension 1 by the proof of \Cref{th:Fuchs-for-mods}.

    \ref{th:irrdiv-existence.eff} Immediate.

    \ref{th:irrdiv-existence.supp} This is \Cref{th:Irr-is-0}.

    \ref{th:irrdiv-existence.an} Apply part \ref{th:irrdiv-existence.supp} to the definition of regularity in the analytic category.

    \ref{th:irrdiv-existence.compl} Apply part \ref{th:irrdiv-existence.supp} to the definition of regularity in the algebraic category.

\end{proof}

\begin{example}
    Continue the notation of \Cref{ex:elemirr}. Then $\irrdiv(\shE^f)$ is the reduced divisor of the divisor of poles of $f$, i.e.\ $\irrdiv(\shE^f) = (\operatorname{div}_\infty(f))_\mathrm{red}$.
\end{example}

\begin{example}\label{ex:gkz110011-ID}
    Continue the notation of \Cref{ex:gkz110011}. Then $\irrdiv(\shM) = 1\cdot V(x_2)$.
\end{example}

\begin{question}
    Motivated by \Cref{ex:gkz110011-ID}, it is natural to ask what $\irrdiv(\shM_A(\beta))$ is for an arbitrary GKZ system. 
\end{question}

\begin{remark}
    One can show that the function $\irrdiv$ is additive and that it commutes with non-characteristic restriction. However, it does not in general commute with restriction to an arbitrary smooth subvariety $Z$, even if $Z$ is not contained in any irreducible component of the pole divisor of $\shM$. For instance, let $X=\PP^n$,  $i\colon Z\hookrightarrow X$ a smooth subvariety, $H\subseteq X$ a hyperplane not containing $Z$, and assume that the divisor $i^*H$ is not reduced. Consider the meromorphic connection $\shM$ on $X$ which is locally given by $\shE^{1/\ell}$, where $\ell$ is a local equation of $H$. Then $\irrdiv(i^*\shM) = (i^*H)_{\mathrm{red}}$, while $i^*\irrdiv(\shM) = i^*H$.
\end{remark}


\bibliographystyle{amsalpha}
\bibliography{bib}

\providecommand{\bysame}{\leavevmode\hbox to3em{\hrulefill}\thinspace}
\providecommand{\MR}{\relax\ifhmode\unskip\space\fi MR }
\providecommand{\MRhref}[2]{%
  \href{http://www.ams.org/mathscinet-getitem?mr=#1}{#2}
}
\providecommand{\href}[2]{#2}
\begin{thebibliography}{DMM12}

\bibitem[Cou95]{coutinho}
S.~C. Coutinho, \emph{A primer of algebraic {$D$}-modules}, London Mathematical Society Student Texts, vol.~33, Cambridge University Press, Cambridge, 1995. \MR{1356713}

\bibitem[Del70]{deligneRH}
Pierre Deligne, \emph{\'{E}quations diff\'{e}rentielles \`a points singuliers r\'{e}guliers}, Lecture Notes in Mathematics, vol. Vol. 163, Springer-Verlag, Berlin-New York, 1970. \MR{417174}

\bibitem[DMM12]{nilssonAhyp}
Alicia Dickenstein, Federico~N. Mart\'{\i}nez, and Laura~Felicia Matusevich, \emph{Nilsson solutions for irregular {$A$}-hypergeometric systems}, Rev. Mat. Iberoam. \textbf{28} (2012), no.~3, 723--758. \MR{2949617}

\bibitem[GS]{M2}
Daniel~R. Grayson and Michael~E. Stillman, \emph{Macaulay2, a software system for research in algebraic geometry}, Available at \url{http://www2.macaulay2.com}.

\bibitem[Hot98]{hotta1998equivariant}
Ryoshi Hotta, \emph{Equivariant d-modules}, 1998.

\bibitem[HTT08]{htt}
Ryoshi Hotta, Kiyoshi Takeuchi, and Toshiyuki Tanisaki, \emph{{$D$}-modules, perverse sheaves, and representation theory}, Progress in Mathematics, vol. 236, Birkh\"auser Boston, Inc., Boston, MA, 2008, Translated from the 1995 Japanese edition by Takeuchi. \MR{2357361}

\bibitem[Inc44]{ince}
E.~L. Ince, \emph{Ordinary {D}ifferential {E}quations}, Dover Publications, New York, 1944. \MR{10757}

\bibitem[KK79]{KKshort}
Masaki Kashiwara and Takahiro Kawai, \emph{On holonomic systems with regular singularities}, Seminar on {M}icrolocal {A}nalysis, Ann. of Math. Stud., vol. No. 93, Princeton Univ. Press, Princeton, NJ, 1979, pp.~113--121. \MR{560314}

\bibitem[Meb89]{Meb89}
Zoghman Mebkhout, \emph{Le th\'{e}or\`eme de comparaison entre cohomologies de de {R}ham d'une vari\'{e}t\'{e} alg\'{e}brique complexe et le th\'{e}or\`eme d'existence de {R}iemann}, Inst. Hautes \'{E}tudes Sci. Publ. Math. (1989), no.~69, 47--89. \MR{1019961}

\bibitem[Meb04]{Meb04}
\bysame, \emph{Le th\'{e}or\`eme de positivit\'{e}, le th\'{e}or\`eme de comparaison et le th\'{e}or\`eme d'existence de {R}iemann}, \'{E}l\'{e}ments de la th\'{e}orie des syst\`emes diff\'{e}rentiels g\'{e}om\'{e}triques, S\'{e}min. Congr., vol.~8, Soc. Math. France, Paris, 2004, pp.~165--310. \MR{2077649}

\bibitem[Sab93]{sab93}
Claude Sabbah, \emph{Introduction to algebraic theory of linear systems of differential equations}, \'{E}l\'{e}ments de la th\'{e}orie des syst\`emes diff\'{e}rentiels. {$\scr D$}-modules coh\'{e}rents et holonomes ({N}ice, 1990), Travaux en Cours, vol.~45, Hermann, Paris, 1993, pp.~1--80. \MR{1603680}

\bibitem[SST00]{sst}
Mutsumi Saito, Bernd Sturmfels, and Nobuki Takayama, \emph{Gr\"{o}bner deformations of hypergeometric differential equations}, Algorithms and Computation in Mathematics, vol.~6, Springer-Verlag, Berlin, 2000. \MR{1734566}

\bibitem[SW08]{SW-irreg-gkz}
Mathias Schulze and Uli Walther, \emph{Irregularity of hypergeometric systems via slopes along coordinate subspaces}, Duke Math. J. \textbf{142} (2008), no.~3, 465--509. \MR{2412045}

\end{thebibliography}
\end{document}